\documentclass[11pt,reqno]{amsart}

\usepackage{commandearticle}

\makeatletter
\def\subsection{\@startsection{subsection}{2}%
	\z@{.5\linespacing\@plus.7\linespacing}{.25\linespacing}%
	{\normalfont\bfseries}}
\def\subsubsection{\@startsection{subsubsection}{3}%
	\z@{.5\linespacing\@plus.7\linespacing}{.25\linespacing}%
	{\normalfont\itshape}}
\makeatother
\graphicspath{{Images/}}

\everymath{\displaystyle}
\begin{document}

\begin{abstract}
In this contribution, a new class of lattice Boltzmann schem\-es is introduced and studied. 
 These schemes are presented in a framework that generalizes the multiple relaxation times method of d'Humi\-\`eres. They extend also the Geier's cascaded method. The relaxation phase takes place in a moving frame involving a set of moments depending on a given relative velocity field. We establish with the Taylor expansion method that the equivalent partial differential equations are identical to the ones obtained with the multiple relaxation times method up to the second order accuracy. The method is then performed to derive the equivalent equations up to third order accuracy. 
\end{abstract}

\title[Lattice Boltzmann Schemes with relative velocities]{Lattice Boltzmann Schemes\\ 
with relative velocities}

\author{François Dubois}
\address[François Dubois]{CNAM Paris, Laboratoire de mécanique des structures et des systèmes couplés et Univ. Paris-Sud, Laboratoire de math\'ematiques, UMR 8628, Orsay, F-91405, CNRS, Orsay, F-91405} 
\email{Francois.Dubois@math.u-psud.fr}

\author{Tony Fevrier}
\address[Tony Fevrier]{Univ. Paris-Sud, Laboratoire de math\'ematiques, UMR 8628, Orsay, F-91405, CNRS, Orsay, F-91405}
\email[corresponding author]{Tony.Fevrier@math.u-psud.fr}

\author{Benjamin Graille}
\address[Benjamin Graille]{Univ. Paris-Sud, Laboratoire de math\'ematiques, UMR 8628, Orsay, F-91405, CNRS, Orsay, F-91405}
\email{Benjamin.Graille@math.u-psud.fr}


\maketitle

\section*{Introduction}
\label{sec:S1}

The lattice Boltzmann method (LBM) is a numerical method able to simulate many various hydrodynamic systems: fluids flows \cite{Chen:1998:0,dHu:1994:0,LalLuo:2000:0}, acoustics \cite{Ricot:2009:1}, heat equation \cite{Mezrhab:2008:0}, multiphase fluids \cite{Dellar:2002:0,Yeomans:2000:0}, Schrödinger equation  \cite{ Zhong:2006:0}, for instance. 
It was introduced to overcome some drawbacks of the ``lattice gas automata'' \cite{Frisch:1986:0,Wolfram:1986:0}.
It originates from a discretization of the Boltzmann equation \cite{Qian:1992:0} and comes in two successive phases: an exact transport step and a relaxation step.
The latter is determined by the choice of a simplified Boltzmann collision kernel. The Bhatnagar Gross Krook approach (BGK) \cite{Benzi:1992:0} and the multiple relaxation times approach (MRT) \cite{LalLuo:2000:0} are two usual choices that  lead to a diagonal relaxation phase with a fixed set of moments. 
Despite the variety of applications and domains involving the lattice Boltzmann method, some theoretical aspects remain debatable as the stability at low viscosities and the lack of Galilean invariance. 

 In 2006, Geier proposed the ``cascaded Boltzmann automata'' \cite{Geier:2006:1,Geier:2006:0}. According to \cite{Geier:2006:1}, the lack of Galilean invariance is identified as a source of instability due to the creation of a negative numerical viscosity. Note that the equilibrium distributions are directly derived from the maxwellian distributions, not exactly identical to the equilibrium presented in Qian and al \cite{Qian:1992:0}. 
 Some works deal with the cascaded approach: it is written as a MRT with a generalized equilibrium in \cite{Asi:2008:0,Premnath:2009:0}; forcing terms are included in the cascaded framework and a Chapman-Enskog expansion is presented in \cite{Premnath:2009:0} to recover the Navier-Stokes equations. 
 
The purpose of this paper is to generalize the d'Humières method and also the cascaded method by introducing moments that depend on an arbitrary velocity field. In the following, this approach refers to scheme with relative velocities. 
The collision step happens in a frame moving at a velocity $\wdu$ and not in a fixed frame. 
The velocity field is here considered in a general way: it is an arbitrary function of space and time. Taking $\utilde=0$ reduces the scheme to a classical d'Humières scheme and taking $\utilde$ as the fluid velocity yields to the cascaded Geier's scheme. This method 
provides a ``cascaded triangular'' structure during the transition between the fixed moments of the d'Humières method and the moments depending on the given velocity field.   

Formal expansions are usually performed to characterize the limit problem simulated by the lattice Boltzmann schemes with the acoustic scaling \cite{Dub:2008:0}, or the diffusive scaling \cite{Junk:2005:0}. The Taylor expansion method has been used with the acoustic scaling to derive third order equivalent equations for the general d'Humières schemes \cite{Dub:2009:0}. 
We extend this work to the schemes with relative velocities in the general non linear case. The dependence on space and time of the new velocity field is the main difficulty to circumvent. 

In the first part of this paper we recall the lattice Boltzmann framework and the associated notations. We generalize it and we introduce the schemes with relative velocities. The cascaded scheme then appears as a particular scheme with relative velocities. In the second part, we study  the  third order consistency with the Taylor expansion method: we establish that the resulting hydrodynamic is identical to the one of Lallemand and Luo \cite{LalLuo:2000:0} up to the second order accuracy. At the third order, the additional terms depend explicitly on the given velocity field.

\section{Description of the scheme}
\label{sec:S2}

\subsection{The usual framework}
\label{sec:S21}
We consider $\lattice$, a regular lattice in $d$ dimensions with a typical mesh size $\dx$. The time step $\dt$ is determined by the acoustic scaling after the specification of the velocity scale $\lambda$ by the relation:  
\begin{equation}\label{eq:lambda}
\dt=\dfrac{\dx}{\lambda}.
\end{equation}
For the scheme denoted by \ddqq, we introduce $\vectv=(\vj[0],\ldots,\vj[q-1])$ the set of the $q$ velocities and we assume that for each node $x$ of $\lattice$, and  each $\vj$ in $\vectv$, the point $x+\vj \dt$ is also a node of the lattice $\lattice$.
The aim of the \ddqq scheme is to compute a particle distribution 
$\vectf = (\fj[0],\ldots,\fj[q-1])$
on the lattice $\lattice$ at discrete values of time. It is a numerical scheme to approximate the partial differential equations
\begin{equation*}
\drondt  \fj +\vj \dscal \gradient \fj=-\sum_{k=0}^{q-1}\dfrac{1}{\tauj[jk]}(\fj[k]-\fjeq[k]),
\qquad 0\leq j\leq q-1, 
\end{equation*}
on a grid in space and time
where $\fjeq$ describes the distribution $\fj$ at the equilibrium and $[\tauj[jk]]$ is the relaxation time matrix. The scheme splits into two phases for each time iteration: first, the relaxation that is local in space, and second, the transport for which an exact characteristic method is used.

We use the MRT framework: the relaxation is written into the basis of the moments \break
$\vectm = (\mk[0],\ldots,\mk[q-1])$, 
the invertible matrix of the transformation being denoted by $\MatM$, so that $\vectm=\MatM\vectf$. Classically, we choose $\Mij[kj] = \Pk(\vj)$, where the set $(\Pk[0],\ldots,\Pk[q-1])$ is a linearly independent family of polynomials in the space 
$\R[X_1,\ldots,X_d]$ of polynomials with $d$ variables.
In this contribution, we restrict to $d+1$ conservation laws. The $d+1$ first polynoms associated with these conservation laws are chosen as $1,X_1,\ldots,X_d$. The first moments of the scheme are thus conserved, they are the density and the momentum:
\begin{equation}\label{eq::rhoq}
\rho=\sum_j \fj,\quad q^{\alpha}=\sum_j v_j^{\alpha}\fj,\quad 1\leq\alpha\leq d.
\end{equation}

\begin{remark}
The restriction to $d+1$ conservation laws is not a limitation of this methodology: it could be extended to an arbitrary number of conservation laws. In particular, the case of $d+2$ conservation laws (mass, momentum, and energy) is treated for the d'Humières scheme \cite{Dub:2012:0}. The case of one conservation law for the scheme with relative velocities has also been derived and is going to be published in a further paper.
\end{remark}

Let us now describe one time step of the scheme. The starting point is the density vector $\vectf(x,t)$ in $x\in\lattice$ at time $t$. The moments are computed by 
\begin{equation*}\label{eq:ftom}
\vectm(x,t)=\MatM\vectf(x,t).
\end{equation*}
The relaxation phase then reads
\begin{equation*}\label{eq:relaxation}
\mke(x,t) = \mk(x,t) + \sk (\mkeq(x,t)-\mk(x,t)), \qquad 0\leq k\leq q{-}1,
\end{equation*}
where $\mkeq$, $0\leq k\leq q{-}1$, is the \kieme moment at equilibrium and $\sk$, $0\leq k\leq q{-}1$, the relaxation parameter associated with the \kieme moment. The relaxation parameters of the conserved moments are equal to $0$ by definition.
The densities are then computed by 
\begin{equation*}\label{eq:mtof}
\vectfe(x,t) = \MatMinv\vectme(x,t). 
\end{equation*}
The transport phase finally reads
\begin{equation}\label{eq:transport}
 \fj(x,t+\dt) = \fje(x-\vj\dt,t), \qquad 0\leq j\leq q{-}1.
\end{equation}

\subsection{Presentation of the scheme with relative velocities}
\label{sec:S22}

The scheme we propose to investigate in this paper is a generalization of the usual scheme presented in section~\ref{sec:S21} 
in the spirit of a previous work of Geier \cite{Geier:2006:1,Geier:2006:0}. 
The moments are shifted with respect to a velocity field $\utilde$. This field is assumed to be a given function of space and time.
In other words, we change the matrix $\MatM$---that is a constant matrix for the d'Humières scheme---into $\MatMu$ by defining
\begin{equation}\label{eq:MatMu}
 \Miju[kj] = \Pk(\vj-\utilde), \qquad 0\leq k,j\leq q{-}1.
\end{equation}
The moments $(\mku[0],\ldots,\mku[q-1])$ are then given by
\begin{equation}\label{eq:ftomu}
\vectmu = \MatMu \; \vectf,
\end{equation}
so that $\mku$ is the \kieme moment in the frame moving at the velocity $\utilde$. In the following we call them the ``shifted moments'' in opposition to the ``fixed moments''. 
Let us note that the d'Humières scheme is enclosed in this framework by taking $\utilde=0$.

The matrix $\MatMu$ is obviously invertible for each $\utilde$. 
The vector of particle distributions is thus obtained by
\begin{equation*}\label{eq:mutof}
\vectf = \MatMinvu \; \vectmu.
\end{equation*}
The vector of distribution functions at the equilibrium $\vectfeq$ is chosen independent of the velocity field so that we have
\begin{equation*}
 \vectmequ = \MatMu \vectfeq = \MatMu \MatMinvz \vectmeqz.
\end{equation*}
The relaxation phase is diagonal in the shifted moments basis. 
\begin{equation}\label{eq:relaxationu}
\mkue = \mku{+}\sk (\mkueq{-}\mku), \qquad 0\leq k\leq q{-}1.
\end{equation}
The densities are then computed by 
\begin{equation*}\label{eq:mtof}
\vectfe = \MatMinvu\vectmue. 
\end{equation*}
The transport phase reads
\begin{equation*}
 \fj(x,t+\dt) = \fje(x-\vj\dt,t), \qquad 0\leq j\leq q{-}1.
\end{equation*}

\subsection{Two-dimensional examples: \ddqn}\label{se:example}

In this section, we present two $\ddqn$ schemes with relative velocities: the one associated with the usual orthogonal set of moments \cite{LalLuo:2000:0} and the cascaded Geier scheme \cite{Geier:2006:0}.
For the first one, the set of nine velocities is given by
\begin{equation*}
\vectv=\{(0,0),(\lambda,0),(0,\lambda),(-\lambda,0),(0,-\lambda),(\lambda,\lambda),(-\lambda,\lambda),(-\lambda,-\lambda),(\lambda,-\lambda)\},
\end{equation*}
 where $\lambda\in\R$ is the velocity scale introduced in (\ref{eq:lambda}).
The set of polynomials is the usual orthogonal set \cite{LalLuo:2000:0}:\begin{equation}\label{eq:mom}
\begin{split}
&1,X,Y,\tfrac{1}{\lambda^2}(3(X^2+Y^2)-4\lambda^2),\tfrac{1}{\lambda^2}(X^2-Y^2),\tfrac{1}{\lambda^2}XY,\\
&\tfrac{1}{\lambda^3}X(3(X^2+Y^2)-5\lambda^2),\tfrac{1}{\lambda^3}Y(3(X^2+Y^2)-5\lambda^2),\\
&\tfrac{1}{\lambda^4}(\tfrac{9}{2}(X^2+Y^2)^2-\tfrac{21\lambda^2}{2}(X^2+Y^2)+4\lambda^4).
\end{split}
\end{equation}
The associated matrix of moments $\MatM(0)$ then reads 
\begin{equation}\label{eq:matort}
\MatM(0)=
\begin{pmatrix} 1&1&1&1&1&1&1&1&1\\
0&\lambda&0&-\lambda&0&\lambda&-\lambda&-\lambda&\lambda\\
0&0&\lambda&0&-\lambda&\lambda&\lambda&-\lambda&-\lambda\\
-4&-1&-1&-1&-1&2&2&2&2\\
0&1&-1&1&-1&0&0&0&0\\
0&0&0&0&0&1&-1&1&-1\\
0&-2&0&2&0&1&-1&-1&1\\
0&0&-2&0&2&1&1&-1&-1\\
4&-2&-2&-2&-2&1&1&1&1\\
\end{pmatrix}. 
\end{equation} 
The matrix $\MatM(\wdu)$ is obtained by the multiplication of $\MatM(0)$ by a triangular matrix called the ``shifting matrix'' that is defined by: 
\begin{equation*}\label{eq:defTeq:defT}
 \MatTu = \MatMu \; \MatMinvz.
\end{equation*}
  This matrix converts the fixed moments $\vectmz$ into the shifted moments $\vectmu$.  It can be written as
\begin{equation}\label{eq:matTu}
\MatTu=\begin{pmatrix} 
1 &0   &0 &0 &0\\
A  & I_2& 0 &0 &0 \\
B_1&B_2& I_3 &0&0 \\
C_1 & C_2&C_3 &I_2 &0\\
D_1 &D_2 &D_3& D_4 & 1
\end{pmatrix}. 
\end{equation}
where the matrices $A,~B_1,~B_2,~C_1,~C_2,~C_3,~D_1,~D_2,~D_3,~D_4$ are presented in Appendix \ref{annexe5}.
 \begin{proposition}\label{th:csd2q9}
 {\rm Morphism of groups.}
For the $\ddqn$, the particular choice of polynomials given by {\rm (\ref{eq:mom})} implies that the function $\utilde\mapsto\MatTu$ is a morphism of groups from $(\R^2,+)$ into $(\R^{9\times 9},\cdot)$ so that 
\begin{equation*}\label{eq::morphism}
\MatTuv=\MatTu\MatTv=\MatTv\MatTu,\quad\forall\utilde,\vtilde\in\R^2.
\end{equation*}
\end{proposition}

\begin{proof}
This result is obtained thanks to a formal calculation software. The reader can check it thanks to the source code given in Appendix \ref{annexe5}.
\end{proof}

This property reflects the fact that a change of frame with a velocity $\utilde+\vtilde$ would be equivalent to two successive changes of frame with velocities $\utilde$ and $\vtilde$. It is achieved under a necessary and sufficient condition between the polynomials defining the moments and their derivatives. The $\ddqn$ set of moments particularly verifies this condition.\\

We close this section by examining the case of the cascaded $\ddqn$ Geier's scheme.
This scheme can be viewed as a particular scheme with relative velocities:
 \begin{proposition}\label{th:cascd2q9}
 {\rm The cascaded Geier's scheme is a scheme with relative velocities.}
 The relaxation of the $\ddqn$ cascaded Geier's scheme is diagonal into the basis of moments given by the polynomials
\begin{equation}\label{eq:polgeier} 
1,X,Y,X^2+Y^2,X^2-Y^2,XY,XY^2,X^2Y,X^2Y^2,
\end{equation}
shifted with respect to the fluid velocity $u=q/\rho$ introduced in {\rm (\ref{eq::rhoq})}.
\end{proposition}
\begin{proof}
The details can be found in Appendix \ref{annexe4}.
\end{proof}
In this sense, the schemes with relative velocities are a generalization of the cascaded scheme. It is important to note that the set of moments is not identical to the usual $\ddqn$ one. Indeed, the moments of third and fourth orders are chosen differently and according to \cite{Geier:2006:0}, it may have a role on the stability of the scheme. 

\section{Partial differential equivalent equations}

The purpose of this paragraph is to make a consistency analysis of the scheme with relative velocities. A Taylor expansion method is performed \cite{Dub:2008:0} to derive the third order equivalent equations for the $\ddqq$ scheme with relative velocities, with any general equilibrium. We explicit a formal development of the scheme when $\dt$ and $\dx$ go to zero with $\lambda=\dx/\dt$ fixed (acoustic scaling) and with the relaxation parameters $s$ fixed. We obtain in this way a set of partial differential equations that are consistent with the scheme at different orders. In the following, we do not adapt the Boltzmann scheme to a particular partial differential equation (at first or second order) but we put in evidence the various operators hidden inside the algorithm. The reasoning consists in a formal development of the transport phase (\ref{eq:transport}) at small $\Delta t$, assuming that all particle distributions are the restrictions on the discretized space of a  sufficiently regular distribution function. The Taylor formula can thus be used as much as wanted.
In this section, we particularly prove that the second order equations are the same as the ones of the \ddqq d'Humières scheme \cite{Dub:2008:0}. In particular, for fluid problems, the viscosity appears as proportional to $(1/s_j-1/2)\, \dx$  \cite{Skordos:1993:0,Dub:2008:0,Luo:2011:0}, where $s_j$ is a relaxation parameter. In other words, the approximated physics do not depend on the velocity field $\utilde$ up to the second order. The dependence only appears at the third order on the momentum equation.
 In the following, we note $\Delta=\Delta t$, and $\fj=\fj(x,t)$. The velocity field is denoted by $\wdu=(\wdu^{\alpha})_{1\leq\alpha\leq d}$. To perform the Taylor expansion up to any order $p$, the transport phase (\ref{eq:transport})  is expanded in terms of $\Delta$, with fixed relaxation parameters:
\begin{equation}\label{eq:dvgenf}
\sum_{0\leq l\leq p} \frac{\Delta^l}{l!}\partial^l_t\fj=\sum_{|\frak{a}|\leq p} \frac{(-\Delta)^{|\frak{a}|}}{\frak{a}!}\vj^{\frak{a}}\partial_{\frak{a}}\fje+{\rm \mathcal{O}}(\Delta^{p+1}),\quad 0\leq j \leq p,
\end{equation}
where $\frak{a}=(a_1,\ldots,a_{d})\in\N^d,~ \vj^{\frak{a}}=\smash{\prod_{\alpha=1}^{d}} (\vj^{\alpha})^{a_{\alpha}},~ \partial_{\frak{a}}=\partial_{a_1}\cdots\partial_{a_{d}},~ \frak{a}!=\smash{\prod_{\alpha=1}^{d}}a_{\alpha}!$ and $|\frak{a}|=\sum_{\alpha=1}^{d} a_{\alpha}$.

The relation (\ref{eq:dvgen}) is obtained multiplying (\ref{eq:dvgenf}) by  the matrix $\Miju[kj]$ and summing on $k$. We have for $0\leq k\leq q-1$
\begin{equation}\label{eq:dvgen}
\sum_{\substack{0\leq j\leq q-1, \\0\leq l\leq p}} \frac{\Delta^l}{l!}\Miju\partial^l_t\fj=\sum_{\substack{0\leq j\leq q-1, \\|\frak{a}|\leq p}} \frac{(-\Delta)^{|\frak{a}|}}{\frak{a}!}\vj^{\frak{a}}\Miju\partial_{\frak{a}}\fje+{\rm \mathcal{O}}(\Delta^{p+1}).
\end{equation}
Taking $k=0$ yields to the conservation of the mass, $k\in\{1,\ldots,d\}$ to the conservation of the momentum and $k>d$ to the relations on the non conserved moments. In the following, a sum over a greek parameter goes from $1$ to $d$ and over a latin parameter from $0$ to $q-1$.
\subsection{Zeroth order expansion}

 The Taylor expansion method at zeroth order gives us the following proposition.
 \begin{proposition}\label{th:ordre0} {\rm The state remains close to the equilibrium.}
 \begin{equation*}
\fj=\fjeq+{\rm \mathcal{O}}(\Delta),\  \ \fje=\fjeq+{\rm \mathcal{O}}(\Delta),\quad 0\leq j\leq q-1. 
\end{equation*}   
\end{proposition}

 The proof of this proposition is omitted as it does not involve the velocity field $\utilde$: we refer to \cite{Dub:2008:0}.

\subsection{First order expansion}

\begin{definition}\label{th:defmflux}
The momentum flux in the direction $(\alpha, \beta)$ is defined by:
\begin{equation*}\label{eq:momflux}
\Fab=\sum_{j}^{} v_j^{\alpha}v_j^{\beta}\fjeq,\quad 1\leq\alpha,\beta\leq d.
\end{equation*}
\end{definition}

The further proposition asserts that the scheme is consistent with the Euler equations up to the first order accuracy. 

 \begin{proposition}\label{th:ordre1}{\rm First order equivalent equations.}
The mass and the momentum conservation equations read:
 \begin{gather}
\displaystyle{\partial_t \rho+\sum_{\beta} \partial_{\beta}q^{\beta}= {\rm \mathcal{O}}(\Delta),}\label{eq:masse1}\\
\displaystyle{\partial_t q^{\alpha}+\sum_{\beta}\partial_{\beta}\Fab= {\rm \mathcal{O}}(\Delta),\qquad 1\leq\alpha\leq d}\label{eq:momentum1}.
\end{gather}
\end{proposition}	

\begin{proof}
The expansion (\ref{eq:dvgen}) up to the second order yields:
\begin{equation}\label{eq:dvlpt1}
\mku+\smash{\Delta\sum_j} \Miju[kj]~\partial_t \fjeq=
\mkue-\Delta\sum_{\beta, j} \Miju[kj]~v_j^\beta~ \partial_{\beta}\fjeq+{\rm \mathcal{O}}(\Delta^2). 
\end{equation}

 \textbf{Case 1: $k=0$.} 
The coefficient $\Miju[0j]$ does not depend on $\widetilde{u}$ and the density $\rho$ is conserved by the collision step. Thus (\ref{eq:dvlpt1}) becomes
\begin{equation*}
\displaystyle{\rho+\Delta~ \partial_t \rho=\rho-\Delta\sum_{\beta, j}  v_j^\beta~ \partial_{\beta}\fjeq+{\rm \mathcal{O}}(\Delta^2)}, 
\end{equation*} 
and the mass equation (\ref{eq:masse1}) follows as $v_j$ does not depend on space and time.\\

 \textbf{Case 2: $1\leq k=\alpha\leq d$.}
The momentum $q$ is also conserved. Replacing $k$ by $\alpha$ in (\ref{eq:dvlpt1}) then gives
\begin{equation*}
\displaystyle{\sum_j (v_j^{\alpha}-\widetilde{u}^{\alpha})~\partial_t \fjeq+\sum_{\beta, j}(v_j^{\alpha}-\widetilde{u}^{\alpha})~v_j^\beta~ \partial_{\beta}\fjeq={\rm \mathcal{O}}(\Delta)}, 
\end{equation*}
 and separating the terms corresponding to $v_j^{\alpha}$ and $\wdua$ yields to:
\begin{equation*}
\displaystyle{\partial_t q^{\alpha}+\sum_{\beta, j}\partial_{\beta}(v_j^{\alpha}~v_j^\beta~ \fjeq)-\wdu^{\alpha}(\partial_t \rho+\sum_{\beta} \partial_{\beta}q^{\beta})={\rm \mathcal{O}}(\Delta)}. 
\end{equation*}
The contribution of the velocity field $\utilde$ only involves the first order mass equation (\ref{eq:masse1}) and can be included in the rest ${\rm \mathcal{O}}(\Delta)$.
\end{proof}

We give two definitions useful for the next orders: the particular derivative that is the derivative along a given velocity---this operator naturally  appears several times in the following expansion---and the conservation default that plays an essential role in some transition lemmas. These lemmas are the keys to increase the order in the Taylor expansion.
 
\begin{definition}\label{th:defderdir}
Let $\operatorname{d}_t^j$  be the particular derivative operator defined by \begin{equation}\label{eq:partd}
 \operatorname{d}_t^j=\partial_t+\sum_{\alpha}v_j^{\alpha}\partial_{\alpha},\quad 0\leq j\leq q-1.
 \end{equation}
\end{definition}

\begin{definition}\label{th:defcons}
The conservation default $\thk$ is defined by:
\begin{equation}\label{eq:defcons}
\displaystyle{\thk=\sum_{j}\Miju[kj]~\operatorname{d}_t^j  \fjeq},\quad 0\leq k\leq q-1.
\end{equation}
\end{definition}
The following lemma checks that the conserved moments have no conservation default up to the first order accuracy.
\begin{lemme}\label{th:conscons}
The conservation defaults for the conserved moments read
\begin{equation*}\label{eq:conscons}
\thk={\rm \mathcal{O}}(\Delta), \qquad 0\leq k\leq d.
\end{equation*}
\end{lemme}

\begin{proof}
Actually this proposition is an other way to write the first order equivalent equations. We verify it with the following calculation using (\ref{eq:defcons}).
\begin{align*}
\thk[0]&= \sum_{j}\operatorname{d}_t^j  \fjeq,\\
\thk[\alpha]&=  \sum_{j}v_j^{\alpha}~\operatorname{d}_t^j  \fjeq-\wdu^{\alpha}\sum_{j}\operatorname{d}_t^j  \fjeq,\qquad 1\leq\alpha\leq d.
\end{align*}
Those two quantities are exactly the equation on mass for the first one,
\begin{equation*}
\sum_{j}^{} \operatorname{d}_t^j\fjeq 
=\partial_t\rho+{\rm div}(q),
\end{equation*}
and the equation on momentum for the second one,
\begin{equation*}
\sum_{j}^{} v_j^{\alpha}~\operatorname{d}_t^j\fjeq 
 =\partial_t q^{\alpha}+\sum_{\beta}\partial_{\beta}\Fab.
\end{equation*}
The proof is closed thanks to Prop.~\ref{th:ordre1}.
\end{proof}
The non conserved moments must also be expanded to derive the second order equivalent equations.
The conservation default presented in Def.~\ref{th:defcons} appears in their expansion. 
 
 \begin{proposition}\label{th:transition1}{\rm Second order transition lemma.}
The non conserved moments before and after the collision read
\begin{align*}
\mku& =\mkueq-\frac{\Delta}{s_k}\thk+{\rm \mathcal{O}}(\Delta^2),\quad & d<k\leq q-1,\\
\mkue& =\mkueq+\Big(1-\frac{1}{s_k}\Big)\Delta~\thk+{\rm \mathcal{O}}(\Delta^2),\quad & d< k\leq q-1.
\end{align*}
\end{proposition}	

This proposition is also verified for $k\leq d$. In this case, it only hides the first order expansion developed in Prop.~\ref{th:ordre1}.

\begin{proof}
The second order Taylor expansion of the shifted moments (\ref{eq:dvlpt1}) yields: 
\begin{equation*}
\displaystyle{\mku-\mkue=-\Delta\Big( \sum_j \Miju[kj]~\partial_t \fjeq+\sum_{\beta, j} \Miju[kj]~v_j^\beta~ \partial_{\beta}\fjeq\Big)+\mathcal{O}(\Delta^2)}.
\end{equation*}
 Combining this identity with the relaxation phase (\ref{eq:relaxationu}) gives
 \begin{equation*}
 \displaystyle{\mku=\mkueq-\frac{\Delta}{s_k}\thk+\mathcal{O}(\Delta^2)}.
\end{equation*}
 Replacing in (\ref{eq:relaxationu}) the post collision moments read:
\begin{equation*}
\displaystyle{\mkue=\mkueq+\Big(1-\frac{1}{s_k}\Big)\Delta~\thk+\mathcal{O}(\Delta^2)}. \qedhere
\end{equation*}
\end{proof}

\subsection{Second order expansion}

We first define the momentum velocity tensor that appears associated with the conservation default in the second order equivalent equations.
 \begin{definition}\label{th:tensor}
The momentum velocity tensor is defined by:
\begin{equation}\label{eq:tensor}
\Lau{kp}{l}=\sum_{j=0}^{q-1}\Mijut[kj]\Mijut[pj]\Mijinvu[jl],\qquad 0\leq k,l,p \leq q-1,
\end{equation}
 with \begin{equation*}
\Mijut[kj]=\left\{\begin{split}
&\vj^{\alpha}\qquad \hspace{0.8cm} i\!f~  1\leq k=\alpha\leq d,  \\
&\Miju[kj]\qquad else.
\end{split}\right.
\end{equation*}
\end{definition}	

The following proposition presents the expected expansion. The mass conservation is valid up to the second order accuracy. The momentum equation contains a first order term that depends on the conservation defaults, the momentum velocity tensors and the set of relaxation parameters $s$.

 \begin{proposition}\label{th:ordre2}{\rm Second order equivalent equations.}
The mass and the momentum conservation equations read:
 \begin{gather}
\partial_t \rho+\sum_{\beta} \partial_{\beta}q^{\beta}= {\rm \mathcal{O}}(\Delta^2),\label{eq:masse2}\\
\partial_t q^{\alpha}+\sum_{\beta}\partial_{\beta}\Fab-\Delta\Big(\sum_{\beta,l>d}\!\sigma_l\partial_\beta(\Lau{\alpha\beta}{l} \thk[l])\Big)= {\rm \mathcal{O}}(\Delta^2),\quad 1\leq\alpha\leq d.\label{eq:momentum2}
 \end{gather}
\end{proposition}	

 The proof follows the same steps as the first order case given by Prop.~\ref{th:ordre1} and can be found in Appendix \ref{annexe1}.\\

We now prove that the second order equivalent equations of the $\ddqq$ schemes with relative velocites are independent of the velocity field $\utilde$. Two propositions are required.
 The first one proves that the conservation defaults of the second order moments do not depend on $\utilde$ up to the first order accuracy. The second asserts that the momentum velocity tensors associated with the third and fourth order moments vanish.
 In other words, the only conservation defaults appearing in (\ref{eq:momentum2}) are those associated with the second order moments.The combination of those two propositions gives the expected result.

\begin{proposition}\label{th:conseq}{\rm First order expansion of some conservation defaults.}\newline
For all k  corresponding to a second order moment, ie degree of $P_{k}$ equal to $2$,
 \begin{equation*}\label{eq:conseq}
  \thk=\thz+ {\rm \mathcal{O}}(\Delta). 
  \end{equation*}
\end{proposition}	

\begin{proof}
We choose $k$ corresponding to a second order moment.
The coefficients $\Miju[kj]$ defined by (\ref{eq:MatMu}) are expanded using the exact second order Taylor formula on their associated polynom $\Pk$. 
\begin{equation}\label{eq:pk}
\displaystyle{\Pk(\textbf{X}-\widetilde{u})=\Pk(\textbf{X})-\sum_{\beta=1}^{d}\utilde^{\beta}~\partial_{\beta}\Pk(\textbf{X})+\frac{1}{2}}\sum_{\beta,\gamma=1}^{d}\utilde^{\beta}~\utilde^{\gamma}~\partial^2_{\beta\gamma}\Pk(\textbf{X}).
\end{equation} 
Its derivatives can be also expanded in this way:
\begin{equation}\label{eq:der}
\partial_{\beta}\Pk(\textbf{X})=a^{\beta 0}+\sum_{\gamma=1}^{d}a^{\beta\gamma}X_{\gamma}~ {\rm and}~
\partial_{\beta\gamma}^2\Pk(\textbf{X})=b^{\beta\gamma},
\end{equation} 
for any sets of real constants $\dsp{(a^{\beta\gamma})_{1\leq\beta,\gamma\leq d}}$ and $\dsp{(b^{\beta\gamma})_{1\leq\beta,\gamma\leq d}}$. 
We apply (\ref{eq:pk}) and (\ref{eq:der}) to the velocity $\vj$ to get the expression of $\Miju[kj]$:
\begin{equation*}
\Miju[kj]=\Mijz[kj]-a^{\beta0}\sum_{\beta=1}^d \utilde^{\beta}-\sum_{\beta,\gamma=1}^{d}\utilde^{\beta}~a^{\beta\gamma}~v_j^{\gamma}+\frac{1}{2}
\sum_{\beta,\gamma=1}^{d}\utilde^{\beta}~\utilde^{\gamma}~b^{\beta\gamma}.
\end{equation*}
 Multiplying by $\operatorname{d}_t^j \fjeq$ and summing over $j$, the conservation default yields: 
\begin{equation*}
\begin{split}
\thk
=&~ \thz{-}\Big(a^{\beta0}\sum_{\beta=1}^d \utilde^{\beta}-\frac{1}{2}\sum_{\beta,\gamma=1}^{d}\utilde^{\beta}~\utilde^{\gamma}~b^{\beta\gamma}\Big)\thz[0]-\sum_{\beta,\gamma=1}^{d}\utilde^{\beta}~a^{\beta\gamma}~\thz[\gamma].
\end{split}
\end{equation*}
 The result is given by the lemma \ref{th:conscons} because the two last terms are a linear combination of the conservation defaults of the conserved moments.
\end{proof}

\begin{proposition}\label{th:tensoreq} The first components of the moment velocity tensor do not depend on $\widetilde{u}$.
\begin{equation*}\label{eq:tensoreq}
 \Lau{\alpha\beta}{l}=\Lauz{\alpha\beta}{l},\quad 1\leq\alpha,\beta\leq d,\quad  l\notin\{0,\alpha,\beta\}.
\end{equation*}
\end{proposition}
 
\begin{proof}

We compose the momentum velocity tensor (\ref{eq:tensor}) by $\Miju[lp]$ for any $0\leq p\leq q-1$ and sum over $l$:
\begin{equation}\label{eq:combi}
\sum_{l}\Lau{\alpha\beta}{l}~\Miju[lp]=\vj[p]^{\alpha}\vj[p]^{\beta}.
\end{equation}
 The equation (\ref{eq:combi})  is actually a polynomial identity evaluated on the velocity set: indeed the quantity $v_p^{\alpha}v_p^{\beta}$ is associated with the polynomial $ X_{\alpha}X_{\beta}$ and $\Miju[lp]$ with $P_l$. It induces us to write $ X_{\alpha}X_{\beta}$ into the basis of $\R_2[X_1,\ldots,X_d]$---the space of polynoms with $d$ variables of degree inferior to $2$---chosen to build the \ddqq scheme:
\begin{equation*}
\exists(\apol)\in\R^{qd^2},\quad\displaystyle{X_{\alpha}X_{\beta}=\sum_{l}\apol P_l}.
\end{equation*}
We note that $\apol=0 ~{\rm if~ deg}(P_l)>2$.
We apply this relation to $X=v_p-\widetilde{u}$ and 
expand it using the definition of the matrix of moments (\ref{eq:MatMu}):
\begin{equation*}
\displaystyle{v_p^{\alpha}v_p^{\beta}=(a^{\alpha\beta}_0+\utilde^{\alpha}\utilde^{\beta})\Miju[0p]+\utilde^{\alpha}\Miju[\beta p]+\utilde^{\beta}\Miju[\alpha p]+\sum_{l=1}^{q-1}\apol \Miju[lp]}.
\end{equation*}
The momentum velocity tensors are obtained multiplying by $\Mijinvu[pn]$ for any $0\leq n\leq q-1$ and summing over $p$:
\begin{equation*}
\displaystyle{\Lau{\alpha\beta}{n}=(a^{\alpha\beta}_0+\utilde^{\alpha}\utilde^{\beta})\delta_{0n}+\utilde^{\alpha}\delta_{\beta n}+\utilde^{\beta}\delta_{\alpha n}+\smash{\sum_{l=1}^{q-1}}\apol \delta_{ln}}.
\end{equation*}
 So $$\forall  n\notin\{0,\alpha,\beta\}, ~\forall\widetilde{u},\quad \Lau{\alpha\beta}{n}=a^{\alpha\beta}_{n},$$ that is a real independent of $\widetilde{u}$. 
\end{proof}	
Particularly, the proof shows that $\Lau{\alpha\beta}{l}=0$ for $l$ corresponding to a moment of order greater than three.

\begin{corollary}\label{th:coreqo2}
 The second order equation on momentum does not depend on the velocity field $\wdu$:
\begin{equation*}\label{eq:coreqo2}
\displaystyle{\partial_t q^{\alpha}+\sum_{\beta}\partial_{\beta}\Fab-\Delta\Big(\sum_{\beta,l>d}\sigma_l~\Lauz{\alpha\beta}{l}~ \partial_\beta\thz[l]\Big)= {\rm \mathcal{O}}(\Delta^2)},\quad1\leq\alpha\leq d.
\end{equation*}
\end{corollary}	
Thus, the second order equivalent equations of the $\ddqq$ scheme with relative velocities are the same as the $\ddqq$ d'Humières ones \cite{dHu:1994:0}. As a consequence, these two different schemes are supposed to approach the same limit problem.
To prepare the third order, the expansion of the non conserved moments is needed. 

 \begin{proposition}\label{th:transition2}{\rm Third order transition lemma.}
  The second order expansion of the conservation default for the conserved moments is, 
\begin{equation}\label{eq:conso2}
\thk=\Delta\Big(\sum_{j,l>d}\sigma_l~\Miju[kj]~\dtj\big(\Mijinvu \thk[l]\big)\Big)+{\rm \mathcal{O}}(\Delta^2),\quad 0\leq k\leq d.
\end{equation}
The non conserved moments before and after the collision read,
\begin{align}
\mku&=~\mkueq-\Delta \Big(\frac{1}{2}+\sigma_k\Big)\xi_k(\utilde,\Delta,\sigma)+{\rm \mathcal{O}}(\Delta^3),~ & d< k\leq q-1,\label{eq:mnc1}\\
\mkue&=~\mkueq+\Delta\Big(\frac{1}{2}-\sigma_k\Big)\xi_k(\utilde,\Delta,\sigma)+{\rm \mathcal{O}}(\Delta^3),~ & d< k\leq q-1,\label{eq:mnc2}
\end{align}
 with 
\begin{equation*}
\xi_k(\utilde,\Delta,\sigma)=\thk{-}\Delta \Big(\sum_{j,l>d}\sigma_l~\Miju[kj]~\dtj\big(\Mijinvu \thk[l]\big)\Big).
\end{equation*}
\end{proposition}	

The proof is presented in Appendix \ref{annexe2}.

\subsection{Third order expansion}
\begin{proposition}\label{th:ordre3}{\rm Third order equivalent equations.}
The mass  and the momentum conservation equations read for all $1\leq\alpha\leq d$,
\begin{equation}\label{eq:masse3}
\partial_t \rho+\sum_{\beta} \partial_{\beta}q^{\beta}-\frac{\Delta^2}{12}\Big(\sum_{\alpha,\beta, j}\partial^2_{\alpha\beta}\big(\vj^{\alpha}\vj^{\beta}~\operatorname{d}_t^j\fjeq\big)\Big)= {\rm \mathcal{O}}(\Delta^3),
 \end{equation}

\begin{multline}\label{eq:momentum3}
\partial_t q^{\alpha}+\sum_{\beta}\partial_{\beta}\Fab =  \Delta\Big(\sum_{\beta,l>d}\sigma_l~\partial_\beta\big(\Lauz{\alpha\beta}{l} \thz[l]\big)\Big)\\
 -\Delta^2 \Big(-\frac{1}{6}\sum_{\beta,j}~\vj^{\alpha}\vj^{\beta}\partial_{\beta}(\dtjfc)+\frac{1}{12}\sum_{\beta,\gamma,j} v_j^\alpha v_j^\beta  v_j^\gamma~\partial^2_{\beta\gamma}(\dtjf)\\
+\sum_{\beta,q,l>d,p>d}\sigma_l\sigma_p~ \partial_{\beta}\big(\Lau{\alpha\beta}{l}\Miju[lq]\dtj[q](\Mijinvu[qp]\thk[p])\big)\Big)+{\rm \mathcal{O}}(\Delta^3).
\end{multline}
\end{proposition}

The proof is presented in Appendix \ref{annexe3}. A second order term independent of the velocity field appears in the mass conservation equation. This term can not be canceled because it does not depend on the scheme parameters. We recover the third order equations derived in \cite{Dub:2009:0} for the d'Humières scheme by taking the velocity field equal to $0$. Indeed in (\ref{eq:momentum3}), the momentum velocity tensor $\Lauz{\alpha\beta}{l}$ does not depend on $\utilde$ and pass through the derivatives.  The second line terms, that are constants of $\utilde$, are already present in  \cite{Dub:2009:0}. The last term becomes
$$\sum_{\beta,q,l>d,p>d}\sigma_l\sigma_p \partial_{\beta}\left(\Lauz{\alpha\beta}{l}\Mijz[lq]\Mijinvz[qp](\partial_t+\sum_{\gamma}v_q^{\gamma}\partial_{\gamma})\left(\thz[p]\right)\right),$$
and separating the particular derivative in two sums yields:
$$\sum_{\beta,q,l>d}\sigma_l^2 \partial^2_{t\beta}\left(\Lauz{\alpha\beta}{l}\thz[l]\right)+\sum_{\beta,\gamma,l>d,p>d}\sigma_l\sigma_p \partial^2_{\beta\gamma}\left(\Lauz{\alpha\beta}{l}\Lauz{l\gamma}{p}\thz[p]\right),$$
that gives exactly the two $\sigma$ dependent third order terms of \cite{Dub:2009:0}. 

If we don't make any assumption on the velocity field, there is one new second order term in the momentum equation. It depends on $\utilde$. We also note the presence of the particular derivative $\dtj$ in all the terms of order greater than one.

\section{Conclusion}

The schemes with relative velocities are a generalization of the MRT approach and of the cascaded method. They appear as MRT  schemes whose relaxation happens in a moving frame at a given velocity. This theory does not add expensive steps to the implementation.
A consistency analysis shows that the second order equivalent equations are independent of the velocity field. The velocity field only plays a role for any order greater than three.
The next step consists in studying the issue of stability and determining whether or not, in what kind of regimes, these schemes bring some improvements.

\section*{Acknowledgements}
We would like to thank the referees for their positive remarks which allowed us to improve the paper.

\begin{appendix}
\section{Composition of the matrix (\ref{eq:matTu})}\label{annexe5}
We present here the blocks of the triangular matrix (\ref{eq:matTu}) defining the $\ddqn$ scheme with relative velocity presented in section \ref{se:example}.
\begin{equation*}
\setlength\arraycolsep{0.25 pt}
A=
\begin{pmatrix} 
-\wdun{1}\\
-\wdun{2}
\end{pmatrix}, ~
B_1=
\begin{pmatrix} 
\dsp{\tfrac{3}{\lambda^2}|\wdu|^2}\\
\dsp{\tfrac{1}{\lambda^2}((\wdun{1})^2-(\wdun{2})^2)}\\
\dsp{\tfrac{1}{\lambda^2}\wdun{1}\wdun{2}}
\end{pmatrix}, ~
B_2=
\begin{pmatrix}
\dsp{-\tfrac{6}{\lambda^2}\wdun{1}}&\dsp{-\tfrac{6}{\lambda^2}\wdun{2}}\\
\dsp{-\tfrac{2}{\lambda^2}\wdun{1}}&\dsp{\tfrac{2}{\lambda^2}\wdun{2}}\\
\dsp{-\tfrac{1}{\lambda^2}\wdun{2}}&\dsp{-\tfrac{1}{\lambda^2}\wdun{1}}
\end{pmatrix},
\end{equation*} 

\begin{equation*}
C_1=
\begin{pmatrix} 
\dsp{-\tfrac{3}{\lambda^3}\wdun{1}(|\wdu|^2+\lambda^2)}\\
\dsp{-\tfrac{3}{\lambda^3}\wdun{2}(|\wdu|^2+\lambda^2)}
\end{pmatrix},~
C_2=
\begin{pmatrix} \dsp{\tfrac{3}{\lambda^3}(|\wdu|^2+2(\wdun{1})^2)}&\dsp{\tfrac{6}{\lambda^3}\wdun{1}\wdun{2}}\\
\dsp{\tfrac{6}{\lambda^3}\wdun{1}\wdun{2}}&\dsp{\tfrac{3}{\lambda^3}(|\wdu|^2+2(\wdun{2})^2)}
\end{pmatrix},
\end{equation*} 

\begin{equation*}
C_3=
\begin{pmatrix} 
-\tfrac{2}{\lambda}\wdun{1}&-\tfrac{3}{\lambda}\wdun{1}&-\tfrac{6}{\lambda}\wdun{2}\\
-\tfrac{2}{\lambda}\wdun{2}&\tfrac{3}{\lambda}\wdun{2}&-\tfrac{6}{\lambda}\wdun{1}
\end{pmatrix},
\end{equation*} 

\begin{equation*}
D_1=\dsp{\tfrac{9}{2\lambda^4}(|\wdu|^4+3\lambda^2|\wdu|^2)},~
D_2=
\begin{pmatrix} \dsp{-\tfrac{9}{2\lambda^4}\wdun{1}(\lambda^2+2|\wdu|^2)}&\dsp{-\tfrac{9}{2\lambda^4}\wdun{2}(\lambda^2+2|\wdu|^2)}
\end{pmatrix}, 
\end{equation*}

\begin{equation*}
\setlength\arraycolsep{0.25 pt}
D_3=
\begin{pmatrix} \dsp{\tfrac{6}{\lambda^2}|\wdu|^2}&\dsp{\tfrac{9}{\lambda^2}((\wdun{1})^2-(\wdun{2})^2)}&\dsp{\tfrac{36}{\lambda^2}\wdun{1}\wdun{2}}
\end{pmatrix}, ~
D_4=
\begin{pmatrix} \dsp{-\tfrac{6}{\lambda}\wdun{1}}&\dsp{-\tfrac{6}{\lambda}\wdun{2}}
\end{pmatrix}. 
\end{equation*}

The following code is the source which allows the reader to verify  the components of the matrix $\MatTu$ just above and the proposition \ref{th:csd2q9}.

\begin{verbatim}
restart:
with(LinearAlgebra):

## Scheme velocities:
u := Transpose(Matrix([[0,0],[la,0],[0,la],[-la,0],
[0,-la],[la,la],[-la,la],[-la,-la],[la,-la]])):

## The moments
m := [1, X, Y, 6/la^2*(X^2+Y^2)/2-4,  (X^2-Y^2)/la^2, X*Y/la^2,
 (6/la^3*X*(X^2+Y^2)/2-5/la*X), (6/la^3*Y*(X^2+Y^2)/2-5/la*Y),
 18/la^4*(X^2+Y^2)^2/4+4-21/la^2*(X^2+Y^2)/2];

##The matrices of moments:
d := RowDimension(u): lv := ColumnDimension(u):
M := Matrix(lv,lv):Mu := Matrix(lv,lv):
for k from 1 to nops(m) do
  for l from 1 to lv do
    M[k,l] := simplify(subs({X=u[1,l],Y=u[2,l]},op(k,m)));          
    Mu[k,l] := simplify(subs({X=u[1,l]-ux,Y=u[2,l]-uy},op(k,m)));   
  end do:
end do:
M;
invM:=M^(-1):

#The shifting matrix :
Tu:=unapply(simplify(Mu.invM),(ux,uy)):
Tu(ux,uy);

#A morphism of group is verified:
simplify(Tu(ux+vx,uy+vy)-Tu(ux,uy).Tu(vx,vy));
\end{verbatim}

\section{Proof of proposition \ref{th:cascd2q9} }\label{annexe4}

The aim of this paragraph is to identify the Geier's cascaded method \cite{Geier:2006:0} as a scheme with relative velocities. 
One difference with the d'Humières schemes lays in the choice of the collision operator. 
In the two dimensional case, for a seek of invariance by translation and rotation, the collision operator is chosen as $^t\MatM(0)\vectg,$  where  $\MatM(0)$ is the matrix defined by (\ref{eq:matort}) and $\vectg$ is the unknown vector leading to the post-collision state. The cascaded relaxation step is then given by
\begin{equation}\label{eq:relgeier}
\vectfe=\vectf+~^t\MatM(0)\vectg.
\end{equation}
We point out the fact that $\MatM(0)$ does not correspond to the chosen moments for which an other set of polynomials is used:
$$1,X,Y,X^2,Y^2,XY,XY^2,X^2Y,X^2Y^2.$$
In the cascaded framework  \cite{Geier:2006:0}, the relaxation is written into the frame moving at the fluid velocity $u=q/\rho$.
We note $\vectmug$ and $\MatM_{\!g}(u)$ the shifted vector and matrix of moments associated with this polynomial set using the same notations as in section \ref{sec:S22}. The relation (\ref{eq:relgeier}) can thus be written as 
\begin{equation}\label{eq:relgeier2}
\vectmueg=\vectmug+\MatCu \vectg,
\end{equation} where $\MatCu=\MatM_{\!g}(u)~^t\MatM(0)$. 

To evaluate $\vectg$, Geier determinates the direction to reach the equilibrium state and relaxes it by multiplying by the relaxation parameters $s_k$ for $0\leq k\leq8$. The block triangular structure of the matrix $\MatCu$ is used to invert itself. 
In the case of the \ddqn scheme, this inversion begins by the resolution of a $2{\times}2$ system in order to obtain the first two components of the vector $\vectg$. 
\begin{equation}\label{eq:casc1}
\begin{split}
(\vectmequg-\vectmug)_{_{3,4}}&=\MatCu_{_{3\leq i,j\leq 4}}\begin{pmatrix} s_3&0\\
0&s_4
\end{pmatrix} \begin{pmatrix} g_3\\
g_4
\end{pmatrix}
=\begin{pmatrix} s^+&s^-\\
s^-&s^+
\end{pmatrix}^{\!-1}\MatCu_{_{3\leq i,j\leq 4}} \begin{pmatrix} g_3\\
g_4
\end{pmatrix},
\end{split}
\end{equation}
where $$s^+=\frac{s_3+s_4}{2},\quad s^-=\frac{s_3-s_4}{2}.$$
The other moments are successively relaxed, using the previously calculated post-collision values of $\vectg$, 
\begin{equation}\label{eq:casc2}
s_k(\vectmequg-\vectmug)_k=\sum_{j=1}^k\MatCu_{_{kj}}g_{_{j}}.
\end{equation}
We gather the relations (\ref{eq:relgeier2},\ref{eq:casc1},\ref{eq:casc2}) to obtain the relaxation step: $$\vectmueg=\vectmug+\MatCu\vectg = \vectmug+\Lambda(\vectmequg-\vectmug),$$ where $$\Lambda = \operatorname{diag}\left(\begin{pmatrix} 0&0&0\\
\end{pmatrix}, \begin{pmatrix} s^+&s^-\\
s^-&s^+
\end{pmatrix}, \begin{pmatrix} s_5&s_6&s_7&s_8
\end{pmatrix}\right).$$ 
This block diagonal matrix can be diagonalized. 
The relaxation of the $\ddqn$ cascaded scheme is diagonal into the set of moments corresponding to (\ref{eq:polgeier})
shifted with respect to the fluid velocity $u$. The cascaded scheme is then a lattice Boltzmann scheme with relative velocities as defined in section \ref{sec:S22}.

\section{Proof of proposition \ref{th:ordre2} }\label{annexe1}
At the third order, forall $0\leq k\leq q-1$, the expansion (\ref{eq:dvgen}) yields:
\begin{multline}\label{eq:dvlpt2}
\mku+\Delta\sum_j \Miju~\partial_t \fj+\frac{\Delta^2}{2}\sum_j \Miju~\partial^2_t \fjeq+\mathcal{O}(\Delta^3)\\
=\mkue-\Delta\sum_{\beta, j} \Miju v_j^\beta~ \partial_{\beta}\fje+\frac{\Delta^2}{2}\sum_{\beta,\gamma,j} \Miju~v_j^\beta  v_j^\gamma ~\partial^2_{\beta\gamma}\fjeq+{\rm \mathcal{O}}(\Delta^3).
\end{multline}

\textbf{Case 1: $k=0$.} 
From (\ref{eq:dvlpt2}), we get
\begin{equation}\label{eq:prequation}
\displaystyle{ \partial_t \rho+\sum_{\beta}  \partial_{\beta}\Big(\sum_jv_j^\beta \fje\Big)+\frac{\Delta}{2}\Big(\partial^2_t \rho-\sum_{\beta,\gamma,j} v_j^\beta ~ v_j^\gamma~ \partial^2_{\beta\gamma}\fjeq\Big)={\rm \mathcal{O}}(\Delta^2)}.
\end{equation}
 The proof follows what is known for the d'Humières scheme \cite{Dub:2008:0}: the first order term is equal to $ \partial_t \rho+\operatorname{div}(q)$; due to the formal derivation of Prop.~\ref{th:ordre1}, a wave equation is obtained, 
\begin{equation*}
\displaystyle{\partial^2_t \rho-\sum_{\beta,\gamma,j} v_j^\beta~  v_j^\gamma~ \partial^2_{\beta\gamma}\fjeq={\rm \mathcal{O}}(\Delta)}.
\end{equation*}
This gives the mass conservation (\ref{eq:masse2}).\\

 \textbf{Case 2: $1\leq k=\alpha\leq d$.}
The identity $\Miju[\alpha j]=v_j^{\alpha}-\wdu^{\alpha}$ is applied to (\ref{eq:dvlpt2}).
The term corresponding to $\wdua$ gives the equation (\ref{eq:prequation}) and goes in the rest. 
Thus the momentum equation can be written as
\begin{equation*}
\displaystyle{\partial_t q^\alpha+\sum_{\beta, j} v_j^\alpha~ v_j^\beta~ \partial_{\beta}\fje+\frac{\Delta}{2}\Big(\partial^2_t q^\alpha-\sum_{\beta,\gamma,j} v_j^\alpha~ v_j^\beta~  v_j^\gamma \partial^2_{\beta\gamma}\fjeq\Big)={\rm \mathcal{O}}(\Delta^2)}.
\end{equation*}
We set $\displaystyle{I_1=\sum_{\beta, j} v_j^\alpha~ v_j^\beta ~\partial_{\beta}\fje}$ et $\displaystyle{I_2=\sum_{\beta,\gamma,j} v_j^\alpha ~v_j^\beta~  v_j^\gamma ~\partial^2_{\beta\gamma}\fjeq}$. \\
 
To get the final equations involving only conserved and equilibrium quantities, we need to expand $\fje$ up to the second order. The use of a transition lemma on the non conserved moments is the key to reach this order.\\ 
 
 \textbf{Calculation of $I_1$: }
We have
\begin{equation*}
I_1=\sum_{\beta, j,l} v_j^\alpha~ v_j^\beta ~\partial_{\beta}\left(\Mijinvu~\mkue[l]\right).
\end{equation*}
We use the second order transition lemma presented in Prop.~\ref{th:transition1} to replace $\mkue[l]$,
and we gather terms by order to get:
\begin{equation*}
I_1=\sum_{\beta, j} v_j^\alpha ~v_j^\beta~ \partial_{\beta}\fjeq+\Delta\sum_{\beta,j,l>d}\Big(\frac{1}{2}-\sigma_l\Big)\partial_{\beta}\big(v_j^\alpha ~v_j^\beta~ \Mijinvu~\thk[l]\big)+{\rm \mathcal{O}}(\Delta^2).
\end{equation*}

Expanding the second sum to separate the terms with and without $\sigma_l$, the first order term of (\ref{eq:momentum2}) appears:

 \begin{equation*}
 I_1=~\smash{\sum_{\beta}} \partial_\beta \Fab+\frac{\Delta}{2}~\smash{\sum_{\beta,j,l}}\partial_{\beta}\left(v_j^\alpha~ v_j^\beta~ \Mijinvu~\thk[l]\right)
 -\Delta\sum_{\beta,l>d}\sigma_l\partial_\beta\left(\Lau{\alpha\beta}{l}~\thk[l]\right)+{\rm \mathcal{O}}(\Delta^2).
 \end{equation*}

 Note that the second sum carries on all $l$ because $\thk[l]={\rm \mathcal{O}}(\Delta)$ for all $l\leq d$ thanks to the lemma \ref{th:conscons}.
Using the definition (\ref{eq:defcons}) of $\thk[l]$, the second term is improved thanks to some matricial rearrangements. The third term can not be simplified in a same manner. Indeed the sum depends on $l$ via $\sigma_l$ that prevents the matricial simplifications.

\begin{equation*}
I_1=\sum_{\beta} \partial_\beta \Fab+\frac{\Delta}{2}\sum_{\beta,j}\partial_{\beta}(v_j^\alpha ~v_j^\beta (\partial_t+\sum_{\gamma}v_q^{\gamma}\partial_{\gamma})(\fjeq))
-\Delta\sum_{\beta,l>d}\sigma_l~\partial_\beta(\Lau{\alpha\beta}{l}~\thk[l])+{\rm \mathcal{O}}(\Delta^2).
\end{equation*}

Separating the particular derivative (\ref{eq:partd}) in two terms,
 the momentum flux and $I_2$ naturally appears in the first order part.
\begin{equation*}
I_1=\sum_{\beta} \partial_\beta \Fab+\frac{\Delta}{2}\Big(\partial_t\big(\sum_{\beta}\partial_{\beta}\Fab\big)+\sum_{\beta,\gamma,j}\partial^2_{\beta\gamma}(v_j^\alpha~ v_j^\beta~ v_j^\gamma ~\fjeq)\Big)
-\Delta\sum_{\beta,l>d}\sigma_l~\partial_{\beta}(\Lau{\alpha\beta}{l}~\thk[l])+{\rm \mathcal{O}}(\Delta^2).
\end{equation*}
We replace the momentum flux using Prop.~\ref{th:ordre1} 
and get the expected result. 
\begin{equation*}
\displaystyle{I_1+\frac{\Delta}{2}(\partial^2_t q^\alpha-I_2)=\sum_{\beta} \partial_\beta \Fab-\Delta\Big(\sum_{\beta,l>d}\sigma_l~\partial_{\beta}(\Lau{\alpha\beta}{l}~\thk[l])\Big)+{\rm \mathcal{O}}(\Delta^2)}.
\end{equation*}

\section{Proof of proposition \ref{th:transition2}}\label{annexe2}

The proof follows the same steps as for Prop.~\ref{th:transition1}.
Two different identities are used to develop $\mku-\mkue$: the first expression is given by the relaxation step (\ref{eq:relaxationu}),
 and the second by the Taylor expansion (\ref{eq:dvlpt2}). This expansion yields to: 
\begin{multline}\label{eq:rel}
\mku-\mkue=-\Delta\Big(\smash{\sum_j} \Miju~\partial_t \fj+\smash{\sum_{\beta, j}} \Miju~v_j^\beta~ \partial_{\beta}\fje\Big) \\
 +\frac{\Delta^2}{2}\Big({-}\sum_j \Miju\partial_t^2 \fjeq+\sum_{\beta,\gamma, j} \Miju~v_j^\beta\vj^{\gamma}~ \partial^2_{\beta\gamma}\fjeq\Big)+{\rm \mathcal{O}}(\Delta^3).
\end{multline}
Let us note $$K=\sum_{j} \Miju~\partial_t \fj+\sum_{\beta, j} \Miju~v_j^\beta~ \partial_{\beta}\fje,$$ the first order term in $\Delta$. This term depends on the values of the distributions before and after collision. It needs to be developed to get the quantities at the equilibrium. To do so, we replace those distributions in the moments frame using Prop.~\ref{th:transition1}.
The term of zeroth order is the conservation default. 
\begin{multline*}
K=\thk+\Delta \Big(-\sum_{j,l>d}\Big(\frac{1}{2}+\sigma_l\Big) \Miju~\partial_t \big(\Mijinvu\thk[l]\big)\\
+\sum_{\beta,j,l>d} \Big(\frac{1}{2}-\sigma_l\Big)  \Miju\vj^{\beta}~\partial_{\beta} \big(\Mijinvu\thk[l]\big)\Big)+{\rm \mathcal{O}}(\Delta^2).
\end{multline*}
We separate now the terms with and without $\sigma_l$. The lemma \ref{th:conscons} is applied to the sums independent of $\sigma$: thus these sums carry on all $l$. 
They are completely developed using the definition of the conservation default (\ref{eq:defcons}) and some matricial simplifications occur. The $\sigma$ dependent term is expressed thanks to the particular derivative. 
\begin{multline*}
 K=\thk+\frac{\Delta}{2} \Big(-\sum_j \Miju~\partial_t^2 \fjeq+\sum_{\beta,\gamma, j} \Miju~v_j^\beta\vj^{\gamma}~ \partial^2_{\beta\gamma}\fjeq\Big)\\
 -\Delta \Big(\sum_{j,l>d}\sigma_l \Miju~\dtj\big(\Mijinvu\thk[l]\big)\Big)+\mathcal{O}(\Delta^2).
\end{multline*}
We replace this relation in (\ref{eq:rel}).
\begin{equation*}
\mku-\mkue=-\Delta\Big(\thk{-}\Delta \big(\sum_{j,l>d}\sigma_l~\Miju[kj]~\dtj\big(\Mijinvu \thk[l]\big)\big)\Big)+{\rm \mathcal{O}}(\Delta^3).
\end{equation*}
For $k\leq d$, this relation gives the development (\ref{eq:conso2}).
For $k>d$, the relaxation phase gives the final developments (\ref{eq:mnc1},\ref{eq:mnc2}) of the non conserved moments.

\section{Proof of proposition \ref{th:ordre3} }\label{annexe3}

The fourth order of the equation (\ref{eq:dvgen}) reads:
\begin{multline*}
\mku+\Delta\sum_j \Miju~\partial_t \fj+\frac{\Delta^2}{2}\sum_j \Miju~\partial^2_t \fj+\frac{\Delta^3}{6}\sum_j \Miju~\partial^3_t \fjeq\\
=\mkue-\Delta\sum_{\beta, j} \Miju v_j^\beta~ \partial_{\beta}\fje
+\smash{\frac{\Delta^2}{2}}\sum_{\beta,\gamma,j} \Miju v_j^\beta  v_j^\gamma ~\partial^2_{\beta\gamma}\fje\\
-\smash{\frac{\Delta^3}{6}}\sum_{\beta,\gamma,\xi,j} \Miju v_j^\beta  v_j^\gamma\vj^{\xi} ~\partial^3_{\beta\gamma\xi}\fjeq+{\rm \mathcal{O}}(\Delta^4).
\end{multline*}

 \textbf{Case 1: $k=0$.} 
The same method as for the lower order expansions leads to
\begin{equation}\label{eq:prequation2}
\partial_t \rho+\sum_{\beta}  \partial_{\beta}q^{\beta}+\frac{\Delta}{2}\Big(\partial^2_t \rho-\sum_{\beta,\gamma,j} v_j^\beta ~ v_j^\gamma~ \partial^2_{\beta\gamma}\fje\Big)
+\frac{\Delta^2}{6}\Big(\partial^3_t \rho+\sum_{\beta,\gamma,\xi,j} v_j^\beta ~ v_j^\gamma~\vj^{\xi} \partial^3_{\beta\gamma\xi}\fjeq\Big)={\rm \mathcal{O}}(\Delta^3).
\end{equation}
We want to replace all the non equilibrium quantities and time derivatives. We use Prop.~\ref{th:transition1} on the first order term:
\begin{multline}\label{eq:fje}
\sum_{\beta,\gamma,j} v_j^\beta ~ v_j^\gamma~ \partial^2_{\beta\gamma}\fje
= \sum_{\beta,\gamma,j} v_j^\beta ~ v_j^\gamma~ \partial^2_{\beta\gamma}\fjeq 
+\Delta\sum_{\beta,\gamma,j,l>d} \smash{\Big(\frac{1}{2}-\sigma_l\Big)}v_j^\beta ~ v_j^\gamma~ \partial^2_{\beta\gamma}(\Mijinvu\thk[l])+{\rm \mathcal{O}}(\Delta^2).
\end{multline}
To evaluate $\partial^2_t \rho$, the second order equation on mass (\ref{eq:masse2}) and momentum (\ref{eq:momentum2}) are formally derived:
\begin{equation}\label{eq:ddtr}
\begin{split}
\partial^2_t\rho &=-\sum_{\gamma}\partial_{\gamma}(\partial_t q^{\gamma})+{\rm \mathcal{O}}(\Delta^2),\\
&=\sum_{\beta,\gamma,j} v_j^\beta  v_j^\gamma~ \partial^2_{\beta\gamma}\fjeq-\Delta\Big(\sum_{\beta,\gamma, l>d}\sigma_l~\partial^2_{\beta\gamma}(\Lau{\gamma\beta}{l} \thk[l])\Big)+{\rm \mathcal{O}}(\Delta^2).
\end{split}
\end{equation} 
 We gather the calculations (\ref{eq:fje},\ref{eq:ddtr}) and use the definition of the conservation default (\ref{eq:defcons}) to get, 
\begin{equation}\label{eq:fot}
\partial^2_t \rho ~-\smash{\sum_{\beta,\gamma,j}} v_j^\beta ~ v_j^\gamma~ \partial^2_{\beta\gamma}\fje=-\frac{\Delta}{2}\smash{\sum_{\beta,\gamma,j}} v_j^\beta ~ v_j^\gamma~ \partial^2_{\beta\gamma}\big(\dtjf\big)+{\rm \mathcal{O}}(\Delta^2).
\end{equation}
In a same manner, $\partial^3_t \rho$ can be expanded using the first order equations on mass (\ref{eq:masse1}) and momentum (\ref{eq:momentum1}):
\begin{equation*}
\partial^3_t\rho=\sum_{\beta,\gamma}  \partial^3_{\beta\gamma t} \Fba+{\rm \mathcal{O}}(\Delta).
\end{equation*}
As a consequence, the second order term can be written as
\begin{equation}\label{eq:sot}
\partial^3_t \rho+\smash{\sum_{\beta,\gamma,\xi,j}} v_j^\beta ~ v_j^\gamma~\vj^{\xi} \partial^3_{\beta\gamma\xi}\fjeq 
=\smash{\sum_{\beta,\gamma,j}} v_j^\beta ~ v_j^\gamma~ \partial^2_{\beta\gamma}\left(\dtjf\right)+{\rm \mathcal{O}}(\Delta).
\end{equation}
 Thus  replacing (\ref{eq:fot}) and (\ref{eq:sot}) in (\ref{eq:prequation2}), the expression of the first and second order quantities is given by:
$$-\frac{\Delta^2}{12}\Big(\smash{\sum_{\beta, j}}\partial^2_{\alpha\beta}(\vj^{\alpha}\vj^{\beta}\operatorname{d}_t^j\fjeq)\Big).$$
and the expected result is proven.\\

 \textbf{Case 2: $1\leq k=\alpha\leq d.$}
The following equation on the momentum is obtained,
\begin{multline}\label{eq:impulsion3}
\partial_t q^\alpha+\sum_{\beta, j} v_j^\alpha~ v_j^\beta~ \partial_{\beta}\fje+\frac{\Delta}{2}\Big(\partial^2_t q^\alpha-\sum_{\beta,\gamma,j} v_j^\alpha~ v_j^\beta~  v_j^\gamma \partial^2_{\beta\gamma}\fje\Big)\\
+\frac{\Delta^2}{6}\Big(\partial^3_t q^\alpha+\sum_{\beta,\gamma,\xi,j} v_j^\alpha~ v_j^\beta~  v_j^\gamma~\vj^{\xi}~ \partial^3_{\beta\gamma\xi}\fjeq\Big)={\rm \mathcal{O}}(\Delta^3).
\end{multline}
 We set $\displaystyle{J_1=\sum_{\beta, j} v_j^\alpha~ v_j^\beta ~\partial_{\beta}\fje}$ et $\displaystyle{J_2=\sum_{\beta,\gamma,j} v_j^\alpha ~v_j^\beta~  v_j^\gamma ~\partial^2_{\beta\gamma}\fje}$. \\
\\
$J_1$ must be expanded up to the third order and $J_2$ up to the second one to get the whole third order expansion.\\

\hspace{0.5cm}\textbf{Calculation of $J_1$: }
The third order transition lemma, presented in Prop.~\ref{th:transition2}, is applied
and the terms with and without $\sigma_l$ are separated. The expression is also splitted between the first and second order terms,
\begin{equation*}
\begin{split}
J_1=&~\sum_{\beta} \partial_\beta \Fab+\frac{\Delta}{2}\sum_{\beta,j,l>d}\partial_{\beta}\big(v_j^\alpha~ v_j^\beta~ \Mijinvu~\thk[l]\big)-\Delta\sum_{\beta,l>d}\sigma_l\partial_\beta\big(\Lau{\alpha\beta}{l}~\thk[l]\big)\\
&-\frac{\Delta^2}{2}\sum_{\beta,j,q,l> d,p>d}\sigma_p~\partial_{\beta}\Big(v_j^\alpha~ v_j^\beta~ \Mijinvu\Miju[lq]~\dtj[q]\big(\Mijinvu[qp] \thk[p]\big)\Big)\\
&+\Delta^2\sum_{\beta,q,p>d,l>d}\sigma_l\sigma_p~ \partial_{\beta}\Big( \Lau{\alpha\beta}{l}\Miju[lq]~\dtj[q]\big(\Mijinvu[qp] \thk[p]\big)\Big)+{\rm \mathcal{O}}(\Delta^3).
\end{split}
\end{equation*}
 We consider the second term $$L=\sum_{\beta,j,l>d}\partial_{\beta}\big(v_j^\alpha~ v_j^\beta~ \Mijinvu~\thk[l]\big).$$ We want to have a sum carrying on all $l$ in order to make some matricial simplifications using the definition of the conservation default (\ref{eq:defcons}). $L$ is written as the difference of the sums carrying on all $l$ and $l\leq d$. The first sum is completely developed and simplified using (\ref{eq:defcons}). Using the lemma \ref{th:conscons} on the second sum is inadequate: indeed $L$ needs to be developed up to the second order to get the expected third order.
 This is why we use (\ref{eq:conso2}) on this second sum. 
 
\begin{multline}\label{eq:L}
L=\sum_{\beta,j}\partial_{\beta}\big(v_j^\alpha~ v_j^\beta~\partial_{\beta}(\dtjf)\big)\\
-\Delta\sum_{\beta,j,q, l\leq d,p>d}\!\sigma_p~\partial_{\beta}\Big(v_j^\alpha~ v_j^\beta~ \Mijinvu\Miju[lq]~\dtj[q]\big(\Mijinvu[qp] \thk[p]\big)\Big)
+{\rm \mathcal{O}}(\Delta^2).
\end{multline} 

We keep unchanged the first order term and develop the particular derivative of the zeroth order one.
 The momentum flux appears in the time derivative and is replaced using the second order equivalent equation on momentum (\ref{eq:momentum2}):
\begin{equation}\label{eq:t1L}
\smash{\sum_{\beta,j}}\partial_{\beta}\big(v_j^\alpha~ v_j^\beta~\partial_{\beta}(\dtjf)\big)=
-\partial_t^2 q^{\alpha}+\Delta\Big(\smash{\sum_{\beta,l>d}}\sigma_l~\partial^2_{t\beta}(\Lau{\alpha\beta}{l} \thk[l])\Big)+\sum_{\beta,\gamma,j}v_j^\alpha~ v_j^\beta~\vj^{\gamma}\partial^2_{\beta\gamma}\fjeq.
\end{equation}
One can now simplify $J_1$ considering the calculations (\ref{eq:L},\ref{eq:t1L}). 
\begin{equation*}
\begin{split}
J_1=&~\sum_{\beta} \partial_\beta \Fab+\frac{\Delta}{2}\Big(-\partial_t^2 q^{\alpha}+\sum_{\beta,j}v_j^\alpha~ v_j^\beta~\vj^{\gamma}\partial^2_{\beta\gamma}\fjeq\Big)-\Delta\sum_{\beta,l>d}\sigma_l\partial_\beta\big(\Lau{\alpha\beta}{l}~\thk[l]\big)\\
&-\smash{\frac{\Delta^2}{2}}\sum_{\beta,j,q,\textbf{l},p>d}\sigma_p~\partial_{\beta}\Big(v_j^\alpha~ v_j^\beta~ \Mijinvu\Miju[lq]~\dtj[q]\big(\Mijinvu[qp] \thk[p]\big)\Big)+\frac{\Delta^2}{2}\sum_{\beta,l>d}\sigma_l~\partial^2_{t\beta}\big(\Lau{\alpha\beta}{l} \thk[l]\big)\\
&+\Delta^2\sum_{\beta,q,p>d,l>d}\sigma_l\sigma_p~ \partial_{\beta}\Big( \Lau{\alpha\beta}{l}\Miju[lq]~\dtj[q]\big(\Mijinvu[qp] \thk[p]\big)\Big)+{\rm \mathcal{O}}(\Delta^3).
\end{split}
\end{equation*}
We notice that the two first second order terms can be simplified: indeed after doing some matricial rearrangements,  the second term appears as a piece of the first one.
We get the final expression of $J_1$:
\begin{multline}\label{eq:J1}
J_1=~\sum_{\beta} \partial_\beta \Fab+\frac{\Delta}{2}\Big(-\partial_t^2 q^{\alpha}+\sum_{\beta,\gamma,j}v_j^\alpha~ v_j^\beta~\vj^{\gamma}\partial^2_{\beta\gamma}\fjeq\Big)\\
-\Delta\sum_{\beta,l>d}\sigma_l\partial_\beta\big(\Lau{\alpha\beta}{l}~\thk[l]\big)
-\smash{\frac{\Delta^2}{2}}\sum_{\beta,\gamma,j,l>d}\sigma_l~v_j^\alpha~ v_j^\beta~\vj^{\gamma}\partial^2_{\beta\gamma}\big(\Mijinvu \thk[l]\big)\\
+\Delta^2\sum_{\beta,q,p>d,l>d}\sigma_l\sigma_p~ \partial_{\beta}\Big( \Lau{\alpha\beta}{l}\Miju[lq]~\dtj[q]\big(\Mijinvu[qp] \thk[p]\big)\Big)
+{\rm \mathcal{O}}(\Delta^3).
\end{multline}

 \hspace{0.5cm}\textbf{Calculation of $J_2$: }
The proposition \ref{th:transition1} is used in the moments frame
and the terms with and without $\sigma_l$ are splitted. The conservation default appearing in the term independent of $\sigma$ is replaced by its definition (\ref{eq:defcons}). We obtain,
\begin{multline}\label{eq:J2}
J_2=\sum_{\beta,\gamma,j} v_j^\alpha ~v_j^\beta~  v_j^\gamma ~\partial^2_{\beta\gamma}\fjeq+\frac{\Delta}{2}\sum_{\beta,\gamma,j} v_j^\alpha ~v_j^\beta~  v_j^\gamma ~\partial^2_{\beta\gamma}\big(\dtjf\big)\\
-\Delta\sum_{\beta,\gamma,j,l>d} \sigma_l~ v_j^\alpha ~v_j^\beta~  v_j^\gamma ~\partial^2_{\beta\gamma}\big(\Mijinvu\thk[l]\big)+{\rm \mathcal{O}}(\Delta^2).
\end{multline}
 All the terms are now gathered in (\ref{eq:impulsion3}). We recall this equation: 
\begin{equation*}
 \begin{split}
&\partial_t q^\alpha{+}J_1{+}\frac{\Delta}{2}\big(\partial^2_t q^\alpha{-}J_2\big){+}\frac{\Delta^2}{6}\Big(\partial^3_t q^\alpha{+}\sum_{\beta,\gamma,\xi,j} v_j^\alpha~ v_j^\beta~  v_j^\gamma~\vj^{\xi}~ \partial^3_{\beta\gamma\xi}\fjeq\Big){=}{\rm \mathcal{O}}(\Delta^3).
\end{split}
\end{equation*}
 The first order terms are coming from the first order terms of $J_1$ (\ref{eq:J1}) and the zeroth ones of $J_2$ (\ref{eq:J2}).
 We immediately recover the expected expression:
$$-\sum_{\beta,l>d}\sigma_l~\partial_\beta(\Lau{\alpha\beta}{l} \thk[l]).$$
The second order terms come from the second order expansion of $J_1$ (\ref{eq:J1}), the first one of $J_2$ (\ref{eq:J2}) and from the $\Delta^2/6$ multiplier. Two terms are immediately vanishing. It remains: 
\begin{multline*}
\sum_{\beta,q,p>d,l>d}\sigma_l\sigma_p~ \partial_{\beta}\Big( \Lau{\alpha\beta}{l}\Miju[lq]~\dtj[q]\big(\Mijinvu[qp] \thk[p]\big)\Big)\\
+\frac{1}{6}\Big(\partial^3_t q^\alpha+\sum_{\beta,\gamma,\xi,j} v_j^\alpha~ v_j^\beta~  v_j^\gamma~\vj^{\xi}~ \partial^3_{\beta\gamma\xi}\fjeq\Big)-\frac{1}{4}\sum_{\beta,\gamma,j} v_j^\alpha ~v_j^\beta~  v_j^\gamma ~\partial^2_{\beta\gamma}\big(\dtjf\big).
\end{multline*}
The first term is an expected term: we keep it unchanged.
To treat $\partial^3_t q^\alpha$, the first order equation on the momentum (\ref{eq:momentum1}) is formally derived,
\begin{equation*}
\partial^3_t q^\alpha=-\partial_t^2\Big(\sum_{\beta}\partial_{\beta}\Fab\Big)+{\rm \mathcal{O}}(\Delta)=-\sum_{\beta,j}\vj^{\alpha}\vj^{\beta}\partial^3_{tt\beta}\fjeq+{\rm \mathcal{O}}(\Delta).
\end{equation*}

The quantities independent of $\sigma$ can be rearranged to finally obtain,
\begin{multline*}
\frac{1}{6}\Big(\partial^3_t q^\alpha+\sum_{\beta,\gamma,\xi,j} v_j^\alpha~ v_j^\beta~  v_j^\gamma~\vj^{\xi}~ \partial^3_{\beta\gamma\xi}\fjeq\Big)-\frac{1}{4}\sum_{\beta,\gamma,j} v_j^\alpha ~v_j^\beta~  v_j^\gamma ~\partial^2_{\beta\gamma}\big(\dtjf\big)=
 \\
 -\frac{1}{6}\sum_{\beta,j}~\vj^{\alpha}\vj^{\beta}\partial_{\beta}\big(\dtjfc\big)+\frac{1}{12}\sum_{\beta,\gamma,j} v_j^\alpha ~v_j^\beta~  v_j^\gamma~\partial^2_{\beta\gamma}\big(\dtjf\big).
\end{multline*}
That closes the proof.
\end{appendix}

\bibliographystyle{plain}
\bibliography{Bibliographie}

\end{document}